\documentclass[reqno]{amsart}
\usepackage{graphicx,amsmath,amssymb,curves,color,amsthm}
\usepackage[pdftex]{hyperref}
\begin{document}

\newtheorem{thm}{Theorem}[section]
\newtheorem{thm-non}{Theorem}
\newtheorem{cor}[thm]{Corollary}
\newtheorem{lem}[thm]{Lemma}
\newtheorem{prop}[thm]{Proposition}
\renewcommand\[{\begin{equation}}
\renewcommand\]{\end{equation}}

\theoremstyle{remark}
\newtheorem{notations}[thm]{Notations}
\newtheorem{notation}[thm]{Notation}

\theoremstyle{definition}
\newtheorem{defn}[thm]{Definition}
\newtheorem{defns}[thm]{Definitions}

\theoremstyle{remark}
\newtheorem{rem}[thm]{Remark}
\newtheorem{eg}[thm]{Example}
\newtheorem{ob}[thm]{Observation}
\numberwithin{equation}{section}
% MATH -----------------------------------------------------------
\newcommand{\norm}[1]{\left\Vert#1\right\Vert}
\newcommand{\abs}[1]{\left\vert#1\right\vert}
\newcommand{\set}[1]{\left\{#1\right\}}
\newcommand{\Real}{\mathbb R}
\newcommand{\C}{\mathbb{C}}
\newcommand{\Q}{\mathbb{Q}}
\newcommand{\Pro}{\mathbb{P}}
\newcommand{\eps}{\varepsilon}
\newcommand{\To}{\longrightarrow}
\newcommand{\TTo}{\dashrightarrow}
\newcommand{\BX}{\mathbf{B}(X)}
\newcommand{\Z}{\mathbb{Z}}
\newcommand{\K}{\mathbb{K}}
\newcommand{\tb}{\textbf}
\newcommand{\mc}{\mathcal}
\newcommand{\mr}{\mathrm}
\newcommand{\mf}{\mathbf}
\newcommand{\mb}{\mathbb}

\newcommand{\TC}{(\mathbb{C}^*)}
\newcommand{\TK}{(\mathbb{K}^*)}
\newcommand{\Sev}{\mathrm{Sev}(\De,\delta )}
\newcommand{\TS}{\mathrm{Trop}(\mathrm{Sev})}
\newcommand{\p}{\phi}
\newcommand{\w}{\omega}
\newcommand{\n}{\mathrm{in}_}

\newcommand{\Tr}{\mathrm{Trop}}
\newcommand{\bs}{\boldsymbol}
\newcommand{\bsm}{\boldsymbol{m}}
\newcommand{\SD}{\mathcal{S}(\Deelta)}
\newcommand{\TSD}{\mathcal{T}\mathcal{S}(\De)}
\newcommand{\T}{\mathbb{T}}

% Greek---------------------------------------------------
\newcommand{\ta}{\theta}
\newcommand{\al}{\alpha}
\newcommand{\be}{\beta}
\newcommand{\la}{\lambda}
\newcommand{\La}{\Lambda}
\newcommand{\ph}{\phi}
\newcommand{\ga}{\gamma}
\newcommand{\Ga}{\Gamma}
\newcommand{\de}{\delta}
\newcommand{\De}{\Delta}
\newcommand{\Si}{\Sigma}
\newcommand{\si}{\sigma}
\newcommand{\ep}{\epsilon}
\newcommand{\ul}{\underline}
\newcommand{\ol}{\overline}
\newcommand{\uli}{\underline{i}}
\newcommand{\ulm}{\underline{m}}
\newcommand{\inverse}{^{-1}}

\newcommand{\ad}{\mr{ad}}
\newcommand{\Ann}{\mr{Ann}}
\newcommand{\rank}{\mr{rank}}
\newcommand{\Conv}{\mr{Conv}}

\newcommand{\mfk}{\mathbf{k}}
\newcommand{\mfg}{\mathbf{g}}
\newcommand{\mkg}{\mk{g}}
\newcommand{\mkh}{\mk{h}}
\newcommand{\mcF}{\mathcal{F}}
\newcommand{\mcL}{\mathcal{L}}
\newcommand{\mcX}{\mathcal{X}}
\newcommand{\mcO}{\mathcal{O}}
\newcommand{\mcT}{\mc{T}}

\newcommand{\D}{\Delta}

\newcommand{\WW}{\tilde{W}}
\newcommand{\ww}{\tilde{w}}
\newcommand{\Pic}{\mr{Pic}}
\newcommand{\Div}{\mr{Div}}
\newcommand{\Proj}{\mr{Proj}}
\newcommand{\Spec}{\mr{Spec}}
\newcommand{\SL}{\mr{SL}}
\newcommand{\GL}{\mr{GL}}
\newcommand{\Lie}{\mr{Lie}}
\newcommand{\Supp}{\mr{Supp}}
\newcommand{\Mor}{\mr{Mor}}
\newcommand{\init}{\mr{init}}
\newcommand{\lee}{\langle}
\newcommand{\ree}{\rangle}
\newcommand{\ti}{\tilde}
\newcommand{\wt}{\mr{wt}}
\newcommand{\Ad}{\mr{Ad}}
\newcommand{\Aut}{\mr{Aut}}
\newcommand{\End}{\mr{End}}
\newcommand{\Trop}{\mr{Trop}}

\newcommand{\red}{\bf\color{red}}
\newcommand{\comment}[1] {{\ \color{blue}\scriptsize{#1}\ }}

\setlength{\marginparwidth}{0.7in}

\newcommand{\todo}[2][0.8]{\vspace{1 mm}\par \noindent
\marginpar{\textsc{To Do}} \framebox{\begin{minipage}[c]{#1
\textwidth} \tt #2 \end{minipage}}\vspace{1 mm}\par}
\newcommand{\aha}[2][]{\vspace{1 mm}\par \noindent
\framebox{\begin{minipage}[c]{#1 \textwidth}  #2
\end{minipage}}\vspace{1 mm}\par}
\newcommand{\fix}[1]{\marginpar{\textsc{Fix!}}\framebox{#1}}

%%%%%%%%%%%%%%%%%%%
% Another LaTeX trick from Anda, June 2009: for a small marginal note
%%%%%%%%%%%%%%%%%%%
\newcommand{\mnote}[1]{${}^*$\marginpar{\footnotesize ${}^*$#1}}

%%%%%%%%%%%%%%%%%%%%%%%%%%%%%%%%
%Boxed numbers%

\newcommand*{\boxednumber}[1]{%
      \expandafter\readdigit\the\psimexpr#1\relax\relax
}
\newcommand*{\readdigit}[1]{%
      \ifx\relax#1\else
            \boxeddigit{#1}%
             \expandafter\readdigit
      \fi
}
% Format macro used for every digit, adjust to your liking:
\newcommand*{\boxeddigit}[1]{\fbox{#1}\hspace{-\fboxrule}}

\title[]{Secondary Fans and Tropical Severi Varieties}%

\author{Jihyeon Jessie Yang}
\address{Department of Mathematics and
Statistics\\ McMaster University\\ 1280 Main
Street West\\ Hamilton, Ontario L8S4K1\\ Canada}
\email{jyang@math.mcmaster.ca}
\urladdr{\url{http://www.math.mcmaster.ca/~jyang/}}

\keywords{Secondary fan, tropical Severi variety, coherent subdivisions of polygons, intersection multiplicity}
\subjclass[2000]{Primary:14T05; Secondary: 14N10}

\thanks{The author thanks Alexander Esterov and Eric Katz for useful discussions. 
}

\date {\today}

\begin{abstract}
This article studies the relationship between tropical Severi varieties and secondary fans. In the case when tropical Severi varieties are hypersurfaces
this relationship is very well known; specifically, in this case, a tropical Severi variety of
codimension 1 is a subfan of the corresponding secondary fan. It was expected
for some time that this continues to hold more generally, but Katz found
a counterexample in codimension 2, showing that this relationship is more subtle. 
The two main results in this paper are as follows. The first theorem finds
a simple condition under which a tropical Severi variety cannot be a subfan
of the corresponding secondary fan. The second theorem provides a partial
converse, namely, we find conditions under which a cone of the secondary fan
is fully contained in the tropical Severi variety. As a first application of these results, we also find a combinatorial formula for the tropical intersection multiplicities for secondary fans.

\end{abstract}

\maketitle

\setcounter{tocdepth}{1}
\tableofcontents

%%%%%%%%%%%%%%%%%%%%%%%%%%%%%%%%%%%%%%%%%%%%%%%%
%%Introduction                                %%
%%%%%%%%%%%%%%%%%%%%%%%%%%%%%%%%%%%%%%%%%%%%%%%%
\section{Introduction}

Severi varieties are  classical objects in algebraic geometry  which go back to F. Enriques \cite{Enriques} and F. Severi \cite{Severi}. They are  projective varieties which parameterize  nodal curves on  toric surfaces. Recent developments of tropical geometry suggest that these classical algebro-geometric objects can be viewed with a different perspective, namely the tropicalizations of Severi varieties ({\em tropical Severi varieties}). These are polyhedral objects on which combinatorial tools can be used. In \cite{JYang} the author found certain descriptions of tropical Severi varieties and used them to provide a tropical intersection theoretic-computation of the degrees of Severi varieties.  In particular, to each point on tropical Severi varieties some  subdivisions of  polygons can be assigned and the tropical intersection multiplicities appearing in the computation are all described in terms of these simple objects, namely, subdivisions of polygons (in fact, only triangles and parallelograms are involved.) 

On the other hand, there are very well-known polyhedral fans which parametrize subdivisions of polygons, called {\em secondary fans}. While the constructions of tropical Severi varieties are algebro-geometric, the constructions of secondary fans are purely combinatorial, although secondary fans were introduced to study the discriminantal   varieties in \cite{GKZ} and also have rich connections to algebraic geometry \cite{CLS, GKZ}. 

So there is a very natural question: {\bf how are secondary fans and tropical Severi varieties related?} Initially it had been expected for a while that a tropical Severi variety for a toric surface $X_{\De}$ from a nondegenerate lattice polygon $\De$ is a subfan of the secondary fan of $\mc{A}=\De\cap\Z^2$, the set of all lattice points on $\De$.  In fact, it is straightforward to show that this statement holds true when the tropical Severi variety is a hypersurface (that is, of codimension 1) (Remark \ref{rem:determinant}).  Also more explicit descriptions of tropical Severi varieties of codimension 1 were studied in \cite {DFS, MMS}. However, Katz\cite{Katz} found  a counterexample showing that there is a  tropical Severi variety of codimension 2 on which there does not exist any subfan structure of the corresponding secondary fan. 

This paper provides some answers to the question above for general tropical Severi varieties. In Theorem \ref{thm:1}, we find a simple sufficient condition under which tropical Severi variety cannot be a subfan of the corresponding secondary fan. This implies that the combinatorial object, tropical Severi variety, contains a certain data which comes from algebro-geometric properties of Severi variety. Theorem \ref{thm:2}  addresses the opposite direction. Namely, it describes when a cone in the secondary fan is fully contained in the tropical Severi variety. 

  As a first application of this understanding of the relationship between tropical Severi varieties and secondary fans, in Theorem \ref{thm:3} we obtain a combinatorial formula for the tropical intersections of secondary fans with tropical linear spaces. The proof uses a  result on tropical intersections obtained in the author's previous paper \cite{JYang}

%%%%%%%%%%%%%%%%%%%%%%%%%%%%%%%%%%%%%%%%%%%%%%%%%%%%%%%%
%Preliminaries
%%%%%%%%%%%%%%%%%%%%%%%%%%%%%%%%%%%%%%%%%%%%%%%%%%%%
\section{Secondary fans and tropical Severi varieties} \label{sec:preliminaries}

\subsection{Secondary Fan} We review on the study of secondary fans. The main referencs for this subsection is \cite[\S 7]{GKZ}.    Simply speaking, the secondary fan of a finite subset $\mc{A}$ of the lattice $\Z^{k}$ is a  fan in $\Real^{|\mc{A}|}$ whose cones parameterize  the coherent marked subdivisions of $(\De, \mc{A})$, where $\De$ is the convex hull of $\mc{A}$.  In this paper, we only consider the case of $k=2$ and $\mc{A}$ is the set of all lattice points on a non-degenerate convex lattice polygon $\De$. (that is, the dimension of $\De$ is 2.) Let us fix the   precise definitions for the notions which we will use in this paper. 

 A \emph{marked polygon} is a pair $(\De,
\mc{A})$ where $\De\subset\Real^{2}$ is a  convex lattice 
polygon   and $\mc{A}\subset \Z^2$ is a finite subset of $\De$ 
containing all the vertices of $\De$ so that the
convex hull of $\mc{A}$ coincides with $\De$. We always assume that $\De$ is {\em non-degenerate}, that is, $\mr{dim}(\De)=2$.

Now  let $\mc{A}=\De\cap\Z^2$, the set of all lattice points in $\De$. Then a (marked)
\emph{subdivision} $S$ of $\De$
 is a collection of marked
 polygons\[\{(\De_i, \mc{A}_i): i=1,\dots,m\},\qquad m\in\Z_{>0}\]  such that
 \begin{enumerate}\item each $\mc{A}_i$ is a subset of $\mc{A}$ and
 each
 $\De_i$ is non-degenerate;
 \item any intersection $\De_i\cap \De_j$ is a face (possibly empty) of both $\De_i$ and
 $\De_j$, and \[\mc{A}_i\cap(\De_i\cap \De_j)=\mc{A}_j\cap(\De_i\cap  \De_j);\]
 \item the union of all $\De_i$ coincides with $\De$.
 \end{enumerate}

 A {\em triangulation} of $\De$ is a subdivision of $\De$ such that for all $i$,  $\De_i$ are triangles and $\mc{A}_i=\mr{Vert}(\De_i)$, the set of vertices of $\De_i$. (That is, only vertices of the triangles are marked.)

 Let $S$ and $S'$ be subdivisions of $\De$. We say that $S$ {\em refines} $S'$ if for each $j$ the collection of $(\De_i,\mc{A}_i)$ such that $\De_i\subset \De_j'$ forms a subdivision of $(\De_j',\mc{A}_j')$, where $S=\{(\De_i, \mc{A}_i): i=1,\dots,m\}, S'=\{(\De_j', \mc{A}_j'):j=1,\dots,m'\}, m, m'\in\Z_{>0}$.
 This makes the set of all subdivisions of $\De$ into a poset. Triangulations are precisely minimal elements of this poset.  Also the subdivision $\{(\De,\mc{A})\}$ is the unique maximal element.

 A subdivision $S$ of  $\De$ is called \emph{coherent} if
it can be constructed from a function
$\psi\in\Real^{\mc{A}}$ as follows: Let
$G_{\psi}\subset\Real^3$ be  the convex hull of
the set
\[\{(a, y): \quad y\le\psi(a),\quad a\in\mc{A},\quad y\in\Real\}.\]
The upper boundary of $G_{\psi}$ is the
graph of a concave piecewise-linear function which we call the \emph{concave hull} of $\psi$ and denote by $\mr{cc}(\psi)$,
\[\begin{array}{cccc}
  \mr{cc}(\psi): &\De & \rightarrow & \Real, \\
  & x& \mapsto & \mr{max}\{y:(x,y)\in G_{\psi}\}.
\end{array}\] (The {\em upper boundary} of $G_{\psi}$ is by definition the union of faces of $G_{\psi}$ which do not contain vertical half lines.)
 Then $S$ coincides with \[\De_{\psi}:=\{(\De_i,
\mc{A}_i):\quad i=1,\dots,m\},\quad
m\in\Z_{>0},\] where $\De_i\subset  \De$ are the
domains
 of linearity of $\mr{cc}(\psi)$ and
 $\mc{A}_i\subset \De_i$ consists of all
$a\in\mc{A}\cap \De_i$ such that
$\mr{cc}(\psi)(a)=\psi(a)$ (i.e., the point
$(a,\psi(a))$ lies on the upper boundary of
$G_{\psi}$ and thus it is ``visible'' (or ``marked''))

 Given a coherent subdivision
$S$ of $\De$, let
$\mc{C}(S)\subset\Real^{\mc{A}} $ be
the set of all functions $\psi$ such that
$S$ refines $\De_{\psi}$.

\begin{prop}\cite[\S 7]{GKZ} For any
coherent subdivision $S$ of a non-degenerate convex lattice polygon $\De$, the set $\mc{C}(S)\subset\Real^{\mc{A}}$ is
a closed convex polyhedral cone so that the set of all
such cones is a complete fan in $\Real^{\mc{A}}$. Furthermore, $\mc{C}(S)=\overline{\mc{C}(S)^{\circ}}$, where the relative interior  $\mc{C}(S)^{\circ}$ is the set of all $\psi\in\Real^{\mc{A}}$ sucht that $S$ coincides with $\De_{\psi}$.
\end{prop}

 Note that any cone $\mc{C}(S)$ in the Proposition above contains the line $\Real\cdot \mf{1}$ consisting of the constant functions. The induced  complete fan in the quotient space $\Real^{\mc{A}}/\Real\cdot \mf{1}$  is called the {\em secondary fan} of $\De$ (recall that $\mc{A}=\De\cap\Z^2)$  and denoted by $\Si(\De)$. For simplicity we will use the same notation $\mc{C}(S)$ for the induced cone  in the quotient space  $\Real^{\mc{A}}/\Real\cdot \mf{1}$.

\parbox{8cm}{
\begin{eg}
 The figure on the right illustrates the case when $\De$ is the segment $\mr{Conv}\{0,1,2,3,4\}$, the convex hull of the points $0,1,2,3,4$.
The
 (marked coherent) subdivision
$\De_{\psi}$ consists of
$(\De_1,\{0,2\})$ and $(\De_2,\{2,3,4\})$. Note that the
point $1$ is not marked.
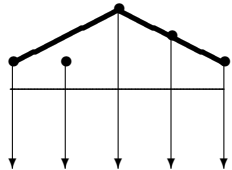
\begin{picture}(0,0)(-210,-40)
\curve(0,0,80,0) \put(-2,8){$\bullet$}
\put(18,8){$\bullet$}\put(38,28){$\bullet$}\put(58,18){$\bullet$}
\put(78,8){$\bullet$}
\put(0,10){$\vector(0,-1){40}$}\put(20,10){$\vector(0,-1){40}$}
\put(40,30){$\vector(0,-1){60}$}\put(60,20){$\vector(0,-1)
{50}$}\put(80,10){$\vector(0,-1){40}$}
\linethickness{0.7mm}
\curve(0,10,40,30)\curve(40,30,80,10)
\end{picture}\\
\end{eg}}

\subsection{Tropical plane curves and connection to secondary fan}\label{sec:tropical plane curve}
In this subsection, we review the definition of a tropical plane curve and find an explicit connection to a secondary fan.

As before let $\De$ be a non-degenerate convex lattice polygon in $\Real^2$ and let $\mc{A}=\De\cap\Z^2$.  Given a function $\psi\in\Real^{\mc{A}}$, the {\em tropical plane curve} $\tau_{\psi}$ with {\em degree} $\De$ is  the
corner locus of the piecewise-linear  function
\[\Real^2\rightarrow \Real, \quad \alpha\mapsto \max_{a\in \mc{A}}\{a\cdot\alpha + \psi(a)\}.\]

Note that the map  $\psi\mapsto \tau_{\psi}$ is not injective. In fact, the tropical plane curve $\tau_{\psi}$ is uniquely
determined by  $\mr{cc}(\psi)_{\mc{A}}$, the concave hull of
$\psi$ (for definition see \S 2.1) restricted on $\mc{A}$. Also it is known that the tropical plane curve $\tau_{\psi}=\tau_{\mr{cc}(\psi)_{\mc{A}}}$ is dual to the subdivision $\De_{\mr{cc}(\psi)}$ of $\De$, as stated in the following Proposition.  Note that the subdivion induced by a {\em concave} function, $\De_{\mr{cc}(\psi)}$, has the property that every lattice point is marked (or visible), that is, $\mc{A}_i=\De_i\cap\Z^2$ for every $i$. 
We introduce several  notions to discribe this type of subdivision in the Definition below. 

\begin{prop}(\cite[\S 2.5.1]{IMS})
The coherent subdivision $\D_{\mr{cc}(\psi)}$ of $\D$ is dual
to the tropical curve  $\tau_{\psi}$ in the
following sense :
\begin{itemize}\item the components of $\Real^2\setminus
 \tau_{\psi}$ are in  1-to-1 correspondence with $\mr{Vert}(\D_{\mr{cc}(\psi)})$;
 \item the edges of $\tau_{\psi}$ are in 1-to-1 correspondence with
 $\mr{Edges}(\D_{\mr{cc}(\psi)})$ so that an edge $e$ of $\tau_{\psi}$ is dual to an
  edge of $\D_{\mr{cc}(\psi)}$ which is orthogonal to $e$ with the lattice length
  equal to the {\em weight} of $e$. (For the definition of the weight of an edge of a tropical curve, see \cite[\S 2.5.1]{IMS});
 \item the vertices of $\tau_{\psi}$ are in 1-to-1 correspondence
 with the 2-dimensional faces of $\D_{\mr{cc}(\psi)}$
  so that the valency of a vertex of $\tau_{\psi}$ is
 equal to the number of sides of the dual face.
\end{itemize}
\end{prop}
    \begin{defns}
  \begin{enumerate}

 \item We say $\psi\in\Real^{\mc{A}}$ is {\em effective} if $\psi=\mr{cc}(\psi)_{\mc{A}}$. 
 \item A subdivision $S$ of $\De$ is called {\em effective} if it is the coherent subdivision $\De_{\psi}$ for some effective element $\psi\in\Real^{\mc{A}}.$ Equivalently, an effective subdivision of $\De$ is a coherent subdivision such that every lattice point is marked, that is, for every $i$, $\mc{A}_i=\De_i\cap\Z^2.$ 
  
\item For a $\psi\in\Real^{\mc{A}}$, define the {\em rank} of $\psi$ to be the dimension of the   cone $\mc{C}(\De_{\mr{cc}(\psi)})$ in the secondary fan $\Si(\De)$  of the effective division $\De_{\mr{cc}(\psi)}$. 

\end{enumerate}
\end{defns}

  Therefore,   given an effective subdivision $S$ of $\De$ we can identify  the cone $\mc{C}(S)^{\circ}$ in the secondary fan $\Si(\De)$ with the set of all tropical plane
curves which are dual to $S$. Thus the {\em effective part} of $\Si(\De)$, the union of such cones for all effective subdivisions, can be identified with the set of  all tropical plane curves with degree $\De$.

\begin{rem}$ $ In \cite{MMS} similar notions are introduced.  The  {\em type} of the subdivision $\De_{\psi}$ (``forgetting all lattice points'') is essentially same as  the effective subdivision $\De_{\mr{cc}(\psi)}$ (``marking all lattice points'').    The rank of $\psi$ is equal to the {\em type dimension} of the type of $\De_{\psi}$.  Also the subdivision  $\De_{\psi}$ is effective if and only if it is of  {\em maximal dimensional type}. 

\end{rem}

\subsection{Tropical plane curves and tropical Severi varieties} In this subsection we review and collect some known results on tropical Severi varieties, which are needed to prove the main theorems in the next section.  Also in this subsection we show the connection between points in tropical Severi varieties and tropical plane curves. 

\subsubsection{Severi variety.} \cite{CH, Fulton2}
As before, let $\De$ be a non-degenerate convex lattice polygon in $\Real^2$.  Denote by  $X_{\De}$  the projective
toric
surface constructed from  $\De$,  $\Pro_{\De}$
 the tautological linear system on $X_{\De}$ and
 $\T_{\De}$  the big open torus of
$\Pro_{\De}$. That is, $\Pro_{\De}$ is the projective space parameterizing curves on the surface $X_{\De}$ defined by polynomials of the following form, \[ f=\sum_{a\in\mc{A}} c_ax^a,\] where $\mc{A}=\De\cap\Z^2$, $x^a$ is the term $x_1^{a_1}x_2^{a_2}$, and the coefficients $c_a$ are taken from the base field. By ordering the elements of $\mc{A}$ we can identify $\Pro_{\De}$ with $\Pro^{n-1}$, where $n=|\mc{A}|$. 

By definition, the {\em Severi variety} $\Sev$ is the (Zariski) closure of the subset of $\Pro_{\De}=\Pro^{n-1}$  consisting of curves with exactly $\de$ nodes (ordinary double points) as their only singularities. It is known that the dimension of $\Sev$ is equal to $n-1-\de$, when the nonnegative integer $\de$ is at most the number of interior lattice points of $\De.$ 
In particular, when $\de=1$, $\mr{Sev}(\De,1)$ is a hypersurface in $\Pro_{\De}$ which is known as $\mc{A}$-discriminantal variety, where $\mc{A}=\De\cap\Z^2$ (\cite{GKZ}).\\

\subsubsection{Tropical Severi variety.}
Now we consider the tropicalization of $\mc{X}=\Sev$, which we call the {\em tropical Severi variety} and denote by $\Tr(\mc{X})$. First, we fix the base field  equipped with a non-Archimedean valuation. In this paper, we use $\K$, the field of locally convergent Puiseux series over $\C$, that is, the element of $\K$ are power series of the form \[c(t)=\sum_{\tau\in R}t^{\tau},\] where $R\subset\Q $ is contained in an arithmetic progression bounded {\em from above}, $c_{\tau}\in \C$ and $\sum_{\tau\in R}|c_{\tau}|t^{\tau}<\infty$ for sufficiently large positive $t$. 
This is an
algebraically closed field of characteristic zero
with a non-Archimedean valuation
\[\mr{Val}(c(t)):=\max\{\tau\in R:c_{\tau}\ne 0\}.\]  

\begin{rem}
The definition of $\K$ given in this paper is different from the standard one in literature, which can be obtained by the change of variable $t\mapsto t\inverse$.  We choose this definition not to have the minus sign in the definition of $\mr{Val}(c(t)).$ 
\end{rem}

Now, by definition $\Tr(\mc{X})$ is the closure of the  image of the following map (Refer to \cite{BG, IMS, Kazarnovskii2, MS} for more details about tropicalization):

\[\mc{X}\cap\T_{\De}\rightarrow \Real^{\mc{A}}/\Real\cdot \mf{1}, \quad [ c_a(t)]_{a\in\mc{A}}\mapsto [ a\mapsto \mr{Val}(c_a(t))],\] where we identify a point in $\mc{X}\cap \T_{\De}$ with the curve defined by the polynomial $f=\sum_{a\in\mc{A}} c_a(t) x^a$ up to scalar multiplications and $\Real\cdot\mf{1}$ is the subspace of $\Real^{\mc{A}}$ consisting of constant functions.  For simplicity, denote the image of $[c_a(t)]_{a\in\mc{A}}$ under this map by $\mr{Val}(f)$.

\subsubsection{Tropical plane curves from tropical Severi variety.}
Now we can attach to $\mr{Val}(f)$ a tropical plane curve and the corresponding effective subdivision of $\De$, \[\tau_f:=\tau_{\mr{Val}(f)},\qquad \De_f:=\De_{\mr{cc}(\mr{Val}(f))}\]  (Note that any representative in $\Real^{\mc{A}}$ of $\mr{Val}(f)$ gives rise to a unique tropical plane curve.) 

Shustin found very nice combinatorial as well as geometric results in this process of tropicalization. Also he found a result which describes the inverse process, called the {\em patchworking}.  To summarize his results we need one more data, namely (tropical) degenerations of the curve defined by $f$, which we describe below: First, we write each $c_a(t)  (a\in\mc{A})$ as follows \[c_a(t)=c_a^{\circ} t^{\mr{cc}(\mr{Val}(f))(a)} + l.o.t.,\] where $c_a^{\circ}$ is some complex number which is zero if $\mr{cc}(\mr{Val}(f))(a)> \mr{Val}(f)(a)$, and $l.o.t.$ stands for ``lower order terms''.  For the effective subdivision $\De_f:\De_1\cup\cdots\cup\De_m$, we have the following collection of complex polynomials (equations for the tropical degenerations):
\[ f_i:=\sum_{a\in\De_i\cap\Z^2} c_a^{\circ} x^a,\quad (i=1,\dots, m).\]

Note that the Newton polygon for each $f_i$ (i.e., the convex hull of $a$ such that $c_a^{\circ}\ne 0$ in $f_i$) is equal to $\De_i$. 

\begin{prop} \label{prop:Shustin} ${}$ 
\begin{enumerate}
 \item \cite[\S 3.3]{Shustin2} {\bf (tropicalization)} Suppose that $\mr{rank}(\tau_f)\ge \mr{dim}(\mc{X})$ for $f\in\mc{X}=\Sev$. Then the corresponding effective subdivision $\De_f$ of $\De$ has the following properties: 
 \begin{enumerate}
 \item (simple) Every boundary lattice point of $\De$ is a vertex of some subpolygon $\De_i, (i=1,\dots,m)$.
 \item (nodal) Every $\De_i (i=1,\dots,m)$ is either a  triangle or a parallelogram.
\end{enumerate}

 (In fact, it is known that the rank of any tropical plane curve $\tau_f$ from $f\in\mc{X}$ is at most $\dim(\mc{X})$. Thus the hypothesis of this statement is equivalent to saying that $\tau_f$ has the maximal rank.)
 
 Also, the complex polynomials $f_1,\dots, f_m$ have the following properties ($\star$):
 
 \begin{enumerate}
 \item the Newton polygon of $f_i$ is equal to $\De_i$.
 \item ($\blacktriangle$) if $\De_i$ is a triangle, the curve defined by $f_i$ is rational and meets the union of toric divisors $X_{\partial \De_i}$ at exactly three points, where it is unibranch.
 \item ($\blacksquare$) if $\De_i$ is a parallelogram, the polynomial $f_i$ has the form \[x^ky^l(\alpha x^{a}+\beta y^b)^p(\gamma x^c+\delta y^d)^q\] with $gcd(a,b)=gcd(c,d)=1, (a:b)\ne(c:d), \alpha,\beta,\gamma,\delta\in\mb{C}\setminus\{0\}$. (gcd stands for the greatest common divisor)
 \item for any common edge $\si=\De_i\cap \De_j$ the truncations $f_i^{\si}$ and $f_j^{\si}$ coincide.

\end{enumerate}•
 \item \cite[\S 5]{Shustin2}{\bf (patchworking)}  Let $\De_{\psi}:\De_1\cup\cdots\cup \De_m$ be an effective  simple nodal subdivision of $\De$ with rank equal to $\dim(\mc{X})$.  Suppose we have a collection of complex polynomials  $F=\{f_1,\dots,f_m\}$ which satisfies the conditions ($\star$) above.  Then there exists $f\in\mc{X}$ with $\mr{cc}(\mr{Val}(f))=\psi$

\end{enumerate}

\end{prop}

\subsection{Tropical intersection and weighted counts of tropical plane curves} \label{sec:intersection} As the last topic of this section, we summarize   some results presented in \cite{JYang} which will be used for the last theorem in the next section. In the tropical intersection theory \cite{AR, Katz2, Kazarnovskii2, JYang}, there is an important intersection multiplicity (which was called an {\em extrinsic intersection multiplicity} in \cite{JYang}). Let us recall it. Suppose  $\psi$ is a transversal intersection  point of two tropical varieties
 $\mc{T}_1, \mc{T}_2$ of complementary
dimension  and so $\mc{T}_i$ is equal to
$L_i$ locally near $\psi, (i=1,2)$, where $L_1$ and
$L_2$ are affine spaces of complementary
dimensions.   
The (extrinsic) intersection
multiplicity of  $\mc{T}_1$ and $\mc{T}_2$ at
$\psi$, denoted by $\xi(\psi;\mc{T}_1,\mc{T}_2)$, is
the volume of the parallelepiped constructed by
the fundamental cells of the lattices
$\mb{L}_i\cap\Z^n,(i=1,2)$ (``principal
parallelepiped''). When  $\mc{T}_1, \mc{T}_2$ are the tropicalizations of complementary dimensional tori, then this multiplicity is in fact equal to the number of intersection points of the tori. 

 In this section, we consider the intersection of the tropical Severi variety $\Tr(\mc{X})=\Tr(\Sev)$ with a complementary dimensional tropical linear space $\Tr(\mc{L}(\bf{p}))$ coming from point conditions. The intersection points of these two spaces correspond to  tropical plane curves passing through the given points counted with certain multiplicities. (For details, see \cite{JYang}.)

\begin{defn}\cite[Definition 2.41]{IMS}\label{defn:GeneralPosition}$ $\begin{enumerate}
\item Let $S$ be an effective  subdivision of $\De$.
We say that the distinct points
$x_1,\dots,x_{\zeta}\in\Q^2$ are in
{\em $S$-general position}, if the condition for
tropical curves with degree $\De$ to pass through
$x_1,\dots,x_{\zeta}$ (``base-point-condition'')
cuts out the  cone $\mc{C}(S)$ either
the empty set, or a polyhedron of codimension
$\zeta$.
\item We say that the distinct points $x_1,\dots,x_{\zeta}$ are in $\De$-{\em general
position (or simply, generic points)}, if they are $S$-general for all effective
subdivisions $S$ of $\De$.
\end{enumerate}
\end{defn}
\begin{lem}\cite[Lemma 2.42]{IMS} For any given convex lattice polygon $\De$, the set of
$\De$-general configurations $x_1,\dots,x_{\zeta}$
is dense in $(\Q^2)^{\zeta}$.
\end{lem}

\begin{defn}\label{defn:linear space}
 Let $\bs{p}=\{p_1,\dots, p_{\zeta}\}\subset \TK^2$ be a finite set of points in $\TK^2$.
 Define $\mc{L}(\bs{p})\subset\Pro_{\De}$ to  be the parameter space of algebraic curves on
 the toric surface $X_{\De}$ passing through all the points in $\bs{p}$. (Remember that the base field is $\mb{K}$.) This parameter space $\mc{L}(\bs{p})$ is
 the complete intersection of hyperplanes $\mc{H}_{p_j}\subset\Pro_{\De}$ defined by the condition of
  passing through the point $p_j,(j=1,\dots, \zeta).$ 
\end{defn}

Now for  $r=\dim(\mc{X})$ let
  $\bs{p}=\{p_1,\dots,p_r\}\subset\TK^2$ be a configuration of  $r$ generic
   points in
  $\TK^2$ so that
$\bs{Val(p)}:=\{Val(p_1),\dots,Val(p_r)\}\subset\Q^2$
 is in $\De$-general position and
  $\mathrm{Trop}(\mc{L}(\bs{p}))\cap\Tr(\mc{X})$
 is a transversal intersection. (For $p_i=(p_{i1}, p_{i2}), Val(p_i):=(Val(p_{i1}), Val(p_{i2})$)  Then for any intersection point $\psi\in\mathrm{Trop}(\mc{L}(\bs{p}))\cap\Tr(\mc{X})$,
 the tropical plane curve $\tau_{\psi}$  passes through all the points in $\bs{Val(p)}$.
Since $\mr{rank}(\psi)=r$,  these points lie  on $r$ distinct edges of
$\tau_{\psi}$ which correspond to some $r$ edges of
the subdivision $\De_{\psi}$. If $\si_i\in
\mr{Edges}(\De_{\psi})$ correspond to a point $Val(p_i)$
and $a_i, a_i'$ are the endpoints of
$\si_i,\quad 1\le i\le r$, then we have the
following linear conditions on $\psi(a_i)$ and
$\psi(a_i')$:
\[\label{independent}\psi(a_i)-\psi(a_i')=(a_i'-a_i)\cdot Val(p_i), \quad(i=1,\dots,r.)\]

We also suppose that  this linear system $\ref{independent}$ is independent.

\begin{prop}\cite[Theorem 4.14]{JYang} \label{prop:Severi degree} With the assumptions as given above,
we can compute the multiplicity as follows:
\[\xi(\psi;\mathrm{Trop}(\mc{L}(\bs{p})),\Tr(\mc{X}))=\frac{\prod
2\mr{area}(\mr{Triangles})}{l(\mb{V}_{S_{\psi}})\cdot
\widetilde{\prod}\mr{length(\mr{Edges})}}, \] where
\begin{enumerate}\item $\prod
2\mr{area}(\mr{Triangles})$ is  the product of  twice the (Euclidean) area of each triangle in
 $\De_{\psi}$.
\item $\widetilde{\prod}\mr{length(\mr{Edges})}$ is the product of the
lattice lengths of the  edges  which are
representatives of each equivalence class in
$\mr{Edges}(\De_{\psi})$, where we define an equivalence
relation as follows: let $e\sim e'$ if $e$ and
$e'$ are the parallel edges of a parallelogram in
$\De_{\psi}$ and extend it by transitivity.

\item $l(\mb{V}_{\De_{\psi}})$ is the number of components of the algebraic set $\mb{V}_{\De_{\psi}}$ defined below.
\end{enumerate}

\end{prop}
\begin{defn}
 Suppose that
$S:=\De_1\cup\cdots\cup \De_m$ is a nodal effective subdivision of $\De$
that is, every sub-polygon is either a triangle
or a parallelogram. Define $\mb{V}_{S}$ to be the set of all
  $f\in\Pro_{\De}$  with the following properties,
\begin{itemize}
 \item ($\blacktriangle$) For every triangle $\De_i$, the curve defined by $f_{\De_i}$ (the truncation of $f$ along $\De_i$) is rational and meets the union of toric divisors $X_{\partial \De_i}$ at exactly three points, where it is unibranch.;
 \item ($\blacksquare$) For every parallelogram $\De_j$,   the polynomial $f_{\De_j}$ has the form \[x^ky^l(\alpha x^{a}+\beta y^b)^p(\gamma x^c+\delta y^d)^q\] with $gcd(a,b)=gcd(c,d)=1, (a:b)\ne(c:d), \alpha,\beta,\gamma,\delta\in\mb{C}\setminus\{0\}$.
\end{itemize}
\end{defn}

The set $\mb{V}_{S}$ is an algebraic set in a certain torus and its number of components can be computed easily using linear algebra. More details can be found in \cite[\S 3]{JYang}.

\section{Main results}\label{sec:proofs} In this last section, we state and prove the main results of this paper, that is, answers to the question: {\bf how are secondary fans and tropical Severi varieties related?}
We use the notations defined in the previous sections. 

The first theorem provides a simple sufficient condition under which the tropical Severi variety, $\Tr(\mc{X})$, cannot be a subfan of the Secondary fan $\Si(\De)$. That is, we cannot find a fan structure on $\Tr(\mc{X})$ such that each cone of $\Tr(\mc{X})$ is the union of some cones of $\Si(\De)$. 

\begin{thm}\label{thm:1}
Suppose that there exists a non-effective  $\psi\in \Tr(\mc{X})$ with the maximal rank,  $\rank(\psi)=\dim(\mc{X})$. Then there is no fan structure on $\Tr(\mc{X})$ which makes $\Tr(\mc{X})$ to be a subfan of the secondary fan $\Si(\De)$. 
\end{thm}

\begin{proof} We consider the following two cones in the secondary fan $\Si(\De)$: 

\begin{itemize}
\item the cone of the subdivision induced by $\psi$, $\mc{C}(\De_{\psi})$;
\item the effective cone of the subdivision induced by $\mr{cc}(\psi)$ (the concave hull of $\psi$), $\mc{C}(\De_{\mr{cc}(\psi)}).$ 
\end{itemize}

Since $\psi$ is non-effective,  $\psi$ is contained in the relative interior of the cone $\mc{C}(\De_{\psi})$ while the cone $\mc{C}(\De_{\mr{cc}(\psi)})$ lies on the boundary of the cone $\mc{C}(\De_{\psi})$.  
By the definition of rank, we know that the dimension of $\mc{C}(\De_{\mr{cc}(\psi)})$ is equal to $r=\dim(\mc{X})$. Thus $\psi$ lies in the relative interior of the cone $\mc{C}(\De_{\psi})$ whose dimension is strictly larger than $r$.  It is known that any fan structure on $\Tr(\mc{X})$ is of dimension $r$. Thus no fan structure on $\Tr(\mc{X})$ can be a subfan of $\Si(\De)$.

\end{proof}

\begin{rem}\label{rem:determinant} Let us consider the case of $\de=1$ for $\mc{X}=\Sev$. As we mentioned before, in this case $\mc{X}$ is the hypersurface of $\Pro_{\De}\cong\Pro^{n-1}$ defined by a polynomial  called {\em $\mc{A}$-discriminant}, where $\mc{A}=\De\cap\Z^2.$ Thus, $\Tr(\mc{X})$ is the codimension one skeleton of  the normal fan $\mcF$ of the Newton polytope of $\mc{A}$-discriminant.
 In \cite{GKZ} it is shown that $\mc{A}$-discriminant is a divisor of another polynomial called {\em the principal $\mc{A}$-determinant} and the normal fan of the Newton polytope of the principal $\mc{A}$-determinant is equal to the secondary fan $\Si(\De)$. Therefore $\mcF$  is a subfan of  the secondary fan $\Si(\De)$. As the codimension one skeleton of $\mcF$, $\Tr(\mc{X})$ is also a subfan of the secondary fan $\Si(\De)$. From the theorem above, we can induce that  there  is no non-effective $\psi$ in $\Tr(\mc{X})$ with the maximal rank. Indeed, this fact is easily seen as follows: Suppose $\psi\in\Tr(\mc{X})$ has the maximal rank, $\rank({\psi}):=\dim(C(\De_{\mr{cc}(\psi)}))=\dim(\mc{X})=n-2$. However $\dim(C(\De_{\mr{cc}(\psi)}))\le\dim(C(\De_{\psi}))\le n-2$. Therefore $C(\De_{\mr{cc}(\psi)})=C(\De_{\psi})$ and so $\psi$ is effective. 
\end{rem}

The following theorem provides another relation between $\Tr(\mc{X})$ and $\Si(\De)$. It is in an opposite direction of the previous theorem in a sense.

\begin{thm}\label{thm:2}
Suppose that there exists an effective cone $\mc{C}$ in $\Si(\De)$ such that $\dim(\mc{C})=\dim(\mc{X})$ and its relative interior $\mc{C}^{\circ}$ intersects with $\Tr(\mc{X})$ at non-zero vectors. Then $\mc{C}^{\circ}$ is fully contained in $\Tr(\mc{X}).$
\end{thm}

\begin{proof}

   Let
$\eta\in\mc{C}^{\circ}\cap\Q^{\mc{A}}$. It is
enough to show that $\n\eta \mc{X}$ (the initial scheme of $\mc{X}$ with respect to $\eta$)  is not empty, which is equivalent to 
$\eta\in\Tr(\mc{X})$. 

Now by the hypothesis given above, we can choose  a non-zero rational  vector $\psi$ contained in $\Tr(\mc{X})\cap\mc{C}^{\circ}$.  Since $\psi\in\Tr(\mc{X})$,  we have a
$f\in \mc{X}$ of the form
\[f(x)=\sum_{a\in \mc{A}}c_a(t)x^a;\quad
c_a(t)=\bar{c_a}t^{\psi(a)}+l.o.t., \quad
\bar{c_a}\in\mb{C}^*.\] Since $\psi$
induces a  concave function,  the tropicalization of $f$ records
\emph{all} $\bar{c_a}$ for $a\in \mc{A}$. Now by
Proposition \ref{prop:Shustin} (2) (patchworking)  we can
obtain a $g\in \mc{X}$ of the form
\[g(x)=\sum_{a\in \mc{A}}c_a'(t)x^a;\quad
c_a'(t)=\bar{c_a}t^{\eta(a)}+l.o.t.\]  Thus
$(\bar{c_a})_{a\in \mc{A}}\in\n\eta\mc{X}$ and
so  $\n\eta \mc{X}$ is not empty.
\end{proof}

Furthermore, by  Proposition \ref{prop:Shustin},  we can obtain the following corollary. 
\begin{cor}\label{cor:sufficient and necessary}

The  cone $\mc{C}(S)^{\circ}$ of an effective subdivision $S$  with $\mr{dim}(\mc{C}(S))=\mr{dim}(\mc{X})$ is fully  contained in $\Tr(\mc{X})$ if and only if  $S$  satisfies the following conditions ($\star\star$) :
\begin{enumerate}
 \item it is simple;
 \item it is nodal;
 \item any parallelogram in $S$ has no special points. (Special points in a parallelogram $\blacksquare$ are lattice points in $\blacksquare$ which are not in the
lattice  generated  by the primitive vectors
 along the sides  of $\blacksquare$.)
\end{enumerate}

\end{cor}

As an application of the results above and the result on the intersections of tropical Severi varieties (\ref{sec:intersection}),  we find a combinatorial formula  in the following for the intersections of secondary fans with   tropical linear spaces.  (Note that the codimension one skeleton of the secondary fan of $\mc{A}$ is equal to the tropicalization of the hypersurface defined by the principal $\mc{A}$-determinant. (see remark \ref{rem:determinant}.)

\begin{thm}\label{thm:3} Let $S$ be an effective subdivision of $\De$  with  dimension $m$ which satisfies the three conditions ($\star\star$) in Corollary \ref{cor:sufficient and necessary}.  Let $H_{q_1},\dots, H_{q_m}$ be the  point-condition tropical hyperplanes for a generic configuration of points $\{q_1,\dots,q_m\}\subset \Q^2$ (The condition of being generic is given in the proof).  Then any point $\psi\in \mc{C}(S)^{\circ}\cap H_{q_1}\cap\cdots\cap H_{q_m}$ occurs with multiplicity equal to
\[\frac{\prod
2\mr{area}(\mr{Triangles})\prod\mr{area}(\mr{Parallelograms})}{l(\mb{V}_{S})\cdot
\prod\mr{length(\mr{Edges})}} \]

\end{thm}

\begin{proof}
 Let $\delta=|\mc{A}|-1-m\ge 0$ so that $\mr{dim}(\mc{X})=m$, where $\mc{X}=\Sev$. Note that $\mc{C}(S)^{\circ}\subset \Tr(\mc{X})$. Now we choose $\bs{p}=\{p_1,\dots,p_m\}\subset(\mb{K}^*)^2$ with $\mr{Val}(\bs{p})=\{q_1,\dots,q_m\}$, which satisfies the hypothesis in Proposition \ref{prop:Severi degree}. Then the intersection multiplicity at $\psi$ is equal to \[\xi(\psi;\mathrm{Trop}(\mc{L}(\bs{p})),\Tr(\mc{X}))=\frac{\prod2\mr{area}(\mr{Triangles})}{l(\mb{V}_{S})\cdot\widetilde{\prod}\mr{length(\mr{Edges})}}. \] Since the subdivision  $S$ has no special points,
\[\frac{\prod2\mr{area}(\mr{Triangles})}{l(\mb{V}_{S})\cdot
\widetilde{\prod}\mr{length(\mr{Edges})}}= \frac{\prod2\mr{area}(\mr{Triangles})\prod\mr{area}(\mr{Parallelograms})}{l(\mb{V}_{S})\cdot
\prod\mr{length(\mr{Edges})}}\] Thus we completed the proof. 
\end{proof}

\bibliographystyle{plain}
%% ***   Set the bibliography file.   ***
%% ("thesis.bib" by default; change if needed)

\end{document}